\documentclass[manuscript,screen]{acmart}
\usepackage{graphicx}
\usepackage{algorithm}
\usepackage{algpseudocode}
\usepackage{subcaption}
\usepackage{bm}
\usepackage{tikz}
\usetikzlibrary{arrows.meta,positioning,shapes.geometric}

\newtheorem{theorem}{Theorem}

\definecolor{blue}{rgb}{0,0,0.9}
\definecolor{red}{rgb}{0.9,0,0}
\definecolor{green}{rgb}{0,0.9,0}
\definecolor{brown}{rgb}{0.6,0.1,0.1}
\definecolor{lightgreen}{rgb}{0.1,0.5,0.1}

\begin{document}

\title{Mixed-Precision GPU Acceleration for Large-Scale Minimum Enclosing Ball Problems}

\author{Ling Liang}
\email{liang.ling@u.nus.edu}
\orcid{0000-0003-0671-9561}
\affiliation{%
  \institution{Department of Mathematics, The University of Tennessee, Knoxville}
  \city{Knoxville}
  \state{TN}
  \country{USA}
}

\author{Lei Yang}
\authornote{Corresponding author.}
\orcid{0000-0001-9979-8371}
\email{yanglei39@mail.sysu.edu.cn}
\affiliation{%
  \institution{School of Computer Science and Engineering, and Guangdong Province Key Laboratory of Computational Science, Sun Yat-sen University}
  \city{Guangzhou}
  \state{Guangdong}
  \country{China}
}

\renewcommand{\shortauthors}{Liang and Yang}

\begin{abstract}
A mixed-precision GPU-oriented optimization framework is developed for computing the minimum enclosing ball of a collection of balls. The approach is built on an equivalent second-order cone programming reformulation and a relative-type inexact proximal augmented Lagrangian method (ripALM), which provides a high-accuracy optimization backbone while solving the inner subproblems only to a progress-dependent relative accuracy. The proximal augmented Lagrangian inherits a constraint-wise separable structure: its objective, gradient, generalized Hessian, and multiplier updates can be efficiently evaluated on GPUs as parallel maps over the input balls followed by reductions. To further improve efficiency, a mixed-precision reduction strategy is introduced. A low-precision ripALM run identifies balls near the approximate boundary, a high-precision ripALM run refines the reduced problem, and a full a posteriori feasibility check detects and reintroduces any violated discarded balls. Thus, low precision is used only for screening and warm starting, while the final feasibility is enforced against the original problem. Experiments on synthetic and real-data instances show that the proposed methods can compute accurate solutions, with GPU execution providing substantial speedups and mixed-precision reduction further reducing running times on most tested instances.
\end{abstract}




\maketitle

\section{Introduction}

Let $\mathcal{B}_i:=\{\bm{x}\in\mathbb{R}^d:\|\bm{x}-\bm{c}_i\|\leq r_i\}$, $i\in[m]:=\{1,\dots,m\}$, be a collection of balls in $\mathbb{R}^d$, where $\bm{c}_i\in\mathbb{R}^d$ and $r_i\geq 0$ denote the center and radius of the $i$th input ball. The minimum enclosing ball (MEB) problem seeks a center $\bm c\in\mathbb R^d$ and the smallest radius $r$ such that the ball centered at $\bm c$ with radius $r$ contains all input balls:
\begin{equation*}\label{eq:meb}\tag{MEB}
\min_{r\in\mathbb{R},\; \bm{c}\in\mathbb{R}^d} \quad r
\qquad\text{s.t.}\qquad
\|\bm{c}-\bm{c}_i\|+r_i\leq r, \quad \forall\,i\in[m].
\end{equation*}
When $r_i=0$ for all $i$, \eqref{eq:meb} reduces to the classical MEB problem for a finite set of points~\cite{welzl1991smallest}. Enclosing-ball models also arise in minimax facility location~\cite{drezner2003facility}, data description and one-class classification~\cite{tax2004support}, kernel learning~\cite{tsang2005core}, and geometric approximation~\cite{badoiu2003smaller,clarkson2010coresets}. In data analysis, individual feature vectors can be represented as points, while groups of observations can be summarized by balls whose radii reflect within-group variation. Large collections of such objects in high-dimensional spaces motivate the development of efficient solvers for \eqref{eq:meb}.

Among existing approaches to \eqref{eq:meb}, geometric methods exploit small support sets to determine the enclosure. Maintaining these sets involves data-dependent updates that complicate GPU parallelization and can become highly costly as $d$ grows \cite{fischer2003fast,gartner1999fast,cgal:eb-26a,welzl1991smallest}. Alternatively, the equivalent second-order cone programming (SOCP) reformulation allows the use of general-purpose conic solvers. For these solvers, however, solving coupled linear systems can dominate the cost of interior-point methods, while first-order splitting methods may require many iterations to attain high accuracy \cite{Clarabel_2024,ocpb:16}. This motivates a structure-exploiting convex optimization framework for computing high-accuracy MEB solutions on GPUs.

We therefore specialize the relative-type inexact proximal augmented Lagrangian method (ripALM) \cite{yang2025convergence,zhu2024ripalm} to the MEB SOCP. Its proximal augmented Lagrangian allows each ball's contribution to the objective, gradient, and generalized Hessian to be evaluated independently, followed by parallel reductions; multiplier updates are also independent across balls. This structure allows the dominant evaluations to be implemented as batched GPU operations. We use the resulting derivatives in a semismooth Newton inner solver \cite{qi1993nonsmooth} whose linear systems involve only the center and radius variables. The relative stopping criterion adapts the required inner accuracy to the progress of the outer iteration.

Even with these computations parallelized, each inner iteration still processes all input balls. For point sets, core-set methods demonstrate that approximate enclosures can be obtained from small representative subsets \cite{badoiu2003smaller,kumar2003minimum}. Motivated by this reduction principle, we develop a mixed-precision ripALM (mp-ripALM). A coarse \texttt{float32} solve identifies balls near or beyond the approximate boundary and provides a warm start for a \texttt{float64} solve on this subset, reducing the constraint-dependent work when screening is effective. Because screening uses an approximate solution, the refined candidate is accepted only after passing a full-set feasibility check at the prescribed tolerance. Discarded constraints violated beyond this tolerance are restored for another reduced solve. Figure~\ref{fig:meb-two-phase-toy} illustrates this workflow.

\begin{figure}[t!]
\centering
\includegraphics[width=0.95\textwidth]{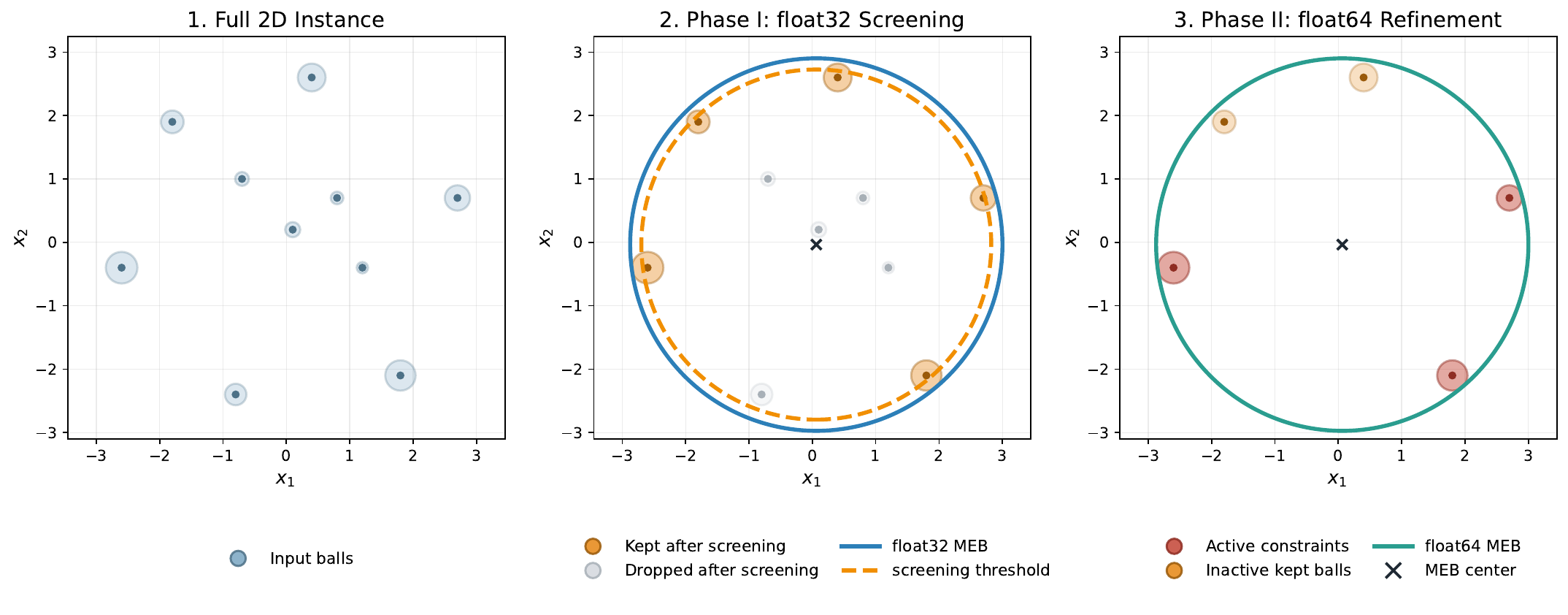}
\caption{Toy two-dimensional example of mp-ripALM. A coarse \texttt{float32} solve identifies near-boundary balls, and a \texttt{float64} solve refines the solution on the retained subset before full-set verification.}\label{fig:meb-two-phase-toy}
\Description{Three panels show the original input balls, low-precision boundary screening, and the refined enclosing ball.}
\end{figure}

The main contributions are fourfold. (i) We specialize ripALM and its existing convergence theory to MEB. (ii) We implement the solver on GPUs using PyTorch and characterize its arithmetic, memory, and synchronization costs. (iii) We develop a mixed-precision screening strategy with a theoretical justification and a full-set feasibility check. (iv) We compare our method with geometric and conic solvers on synthetic data and evaluate it on real data. Further experiments quantify GPU acceleration and assess the effectiveness and limitations of mixed-precision screening.

The rest of this paper is organized as follows. After reviewing related work in Section~\ref{sec:related-work}, we present the SOCP formulation in Section~\ref{sec:socp-background} and specialize ripALM to MEB in Section~\ref{sec:ripalm-meb}. Section~\ref{sec:gpu-implementation} describes the GPU implementation and mixed-precision strategy, followed by numerical experiments in Section~\ref{sec:numerical-experiments} and concluding remarks in Section~\ref{sec:conclusion}.

\paragraph{Notation.} For a positive integer $m$, we write $[m]:=\{1,\ldots,m\}$. The spaces $\mathbb R^n$, $\mathbb R_+^n$, and $\mathbb R_{++}^n$ denote the sets of $n$-dimensional real, nonnegative, and positive vectors, respectively. We use $\mathbb R^{p\times q}$ for the space of real $p\times q$ matrices. The Euclidean inner product and norm are denoted by $\langle\cdot,\cdot\rangle$ and $\|\cdot\|$, respectively. For a closed convex set $\mathcal C$ in a finite-dimensional Euclidean space, $\Pi_{\mathcal C}$ denotes the Euclidean projection onto $\mathcal C$, $\operatorname{dist}(x,\mathcal C):=\inf_{y\in\mathcal C}\|x-y\|$ denotes the distance from $x$ to $\mathcal C$, $\operatorname{int}(\mathcal C)$ denotes its interior, and $\mathcal N_{\mathcal C}(x)$ denotes its normal cone at $x$. For a locally Lipschitz mapping $F$, $\partial F(x)$ denotes its Clarke generalized Jacobian at $x$.

\section{Related work}\label{sec:related-work}

\textbf{Geometric algorithms and software.} A major line of work studies MEB problems from the viewpoint of computational geometry. For point sets, Welzl's randomized incremental algorithm exploits the LP-type structure of the problem and computes an exact solution in expected linear time when the dimension is fixed \cite{welzl1991smallest}. Megiddo's deterministic prune-and-search method gives a linear-time algorithm in a fixed dimension \cite{megiddo1983linear}. This line of work has led to mature software implementations. For example, the \texttt{Miniball} package provides a compact C++ implementation for computing the smallest enclosing ball of a set of points in arbitrary dimension \cite{gartner1999fast}. Fischer et al. \cite{fischer2003fast} develop a pivoting algorithm for high-dimensional point sets. The Computational Geometry Algorithms Library (CGAL) also provides routines for computing the minimum enclosing ball of balls in arbitrary dimension. For its LP algorithm, CGAL reports an expected upper bound of $\mathcal{O}(2^{d}m)$ arithmetic operations \cite{cgal:eb-26a}. From a parallel-computing perspective, data-dependent basis updates in these geometric algorithms can complicate efficient GPU implementation, where performance relies on regular memory access patterns, vectorized arithmetic, and large-scale parallel reductions.

\textbf{General-purpose conic optimization.} Another major algorithmic paradigm for solving \eqref{eq:meb} is based on convex optimization. The MEB problem admits several tractable convex reformulations, including an equivalent second-order cone programming (SOCP) reformulation \cite{boyd2004convex,liang2021inexact}. This conic viewpoint makes it possible to apply interior-point methods (IPMs), which are robust, highly accurate, and supported by polynomial-time complexity theory \cite{nesterov1994interior}. Established solvers include SeDuMi, SDPT3, ECOS, and Clarabel \cite{ECOS2013,Clarabel_2024,sedumi1999,toh1999sdpt3}. For large-scale MEB problems, however, generic IPM implementations are often limited by the cost of assembling and solving Newton linear systems associated with the full conic model; the associated memory and factorization costs can grow substantially with the number of balls $m$ and the data dimension $d$. A related alternative is to apply operator-splitting methods to the SOCP reformulation, as in the conic solver SCS \cite{ocpb:16}. Such methods offer comparatively inexpensive first-order iterations but may require many iterations when high-accuracy feasibility and optimality are required.

\textbf{MEB-specific optimization methods.} A third line of work develops convex-optimization algorithms tailored to MEB. Eliminating $r$ from \eqref{eq:meb} gives the max-type reformulation $\displaystyle \min_{\bm{c}\in\mathbb{R}^{d}}\max_{1\leq i\leq m}\{\|\bm{c}-\bm{c}_{i}\|+r_{i}\}$. Zhou et al. \cite{zhou2005efficient} develop a smoothing-based approach to this problem using log-exponential aggregation and quasi-Newton steps, as well as an algorithm that exploits the special structure of the SOCP formulation. For point sets, Frank--Wolfe-based approaches, including an away-step variant, have been developed for the dual formulation \cite{yildirim2008two}; their iterations use distance evaluations and vector updates. These approaches illustrate how MEB-specific structure can be exploited within convex optimization. However, their efficiency can also deteriorate when a highly accurate solution is required.

\textbf{Core sets and parallel approximate MEB.} For large-scale approximate MEB computation, core-set methods form another influential class of algorithms. For point sets, the seminal work \cite{badoiu2003smaller} showed that a $(1+\varepsilon)$-approximate enclosing ball can be obtained by suitably enlarging the MEB of a small representative subset whose cardinality depends on $\varepsilon$ but not on the number of input points. This core-set principle leads to efficient approximation algorithms for high-dimensional MEB problems with point or ball inputs \cite{kumar2003minimum} and has inspired a broader literature on geometric approximation and core-set constructions \cite{agarwal2005geometric,feldman2011unified}, including sliding-window variants for MEB \cite{wang2019coresets}. In many iterative MEB core-set algorithms, a common operation is to identify a farthest point from the current center. This operation is highly parallelizable and has been accelerated on GPUs using distance filtering techniques to avoid unnecessary distance evaluations \cite{kallberg2014accelerated}. Core-set methods are especially attractive when moderate relative accuracy is sufficient, while tighter approximation requirements generally increase the subset size or iteration count.

\textbf{Inexact augmented Lagrangian methods.} Inexact augmented Lagrangian methods have been developed for SOCP and applied to solve MEB problems \cite{liang2021inexact}. The ripALM framework, with its relative-type inner stopping criterion and convergence analysis, provides the basis for our MEB specialization \cite{yang2025convergence,zhu2024ripalm}. For GPU implementation, we evaluate each ball's contribution to the proximal augmented Lagrangian and its derivatives independently, followed by parallel reductions. Mixed-precision methods in numerical linear algebra combine low- and high-precision computations to balance cost and accuracy \cite{higham2022mixed}. In our framework, low precision is used for screening and warm starting, followed by high-precision refinement on the reduced problem, with full-set feasibility checks and reinsertion of constraints violated beyond the prescribed tolerance.

\section{Background on the SOCP reformulation of \eqref{eq:meb}}\label{sec:socp-background}

We first present the second-order cone programming (SOCP) reformulation on which our algorithmic framework is built. Since the constraint $\|\bm{c}-\bm{c}_i\|+r_i\leq r$ is equivalent to $(r-r_i,\bm{c}-\bm{c}_i)\in\mathcal Q_{d+1}$, the MEB problem can be written as
\begin{equation}\label{eq:meb-socp}
\min_{r\in\mathbb{R},\; \bm{c}\in\mathbb{R}^d}
\qquad  r \qquad
\text{s.t.} \qquad
\begin{pmatrix}
r-r_i \\
\bm{c}-\bm{c}_i
\end{pmatrix}\in\mathcal{Q}_{d+1},
\quad\forall\,i\in[m],
\end{equation}
where $\mathcal{Q}_{d+1}
:=\left\{(t, \bm{s})^{\top}\in\mathbb{R}^{d+1}:\|\bm{s}\|\leq t
\right\}$
is the second-order cone in $\mathbb{R}^{d+1}$. Thus, problem \eqref{eq:meb-socp} is exactly the MEB problem \eqref{eq:meb}, written in conic form without introducing additional variables.

Note that the primal optimal solution of \eqref{eq:meb-socp} (and hence \eqref{eq:meb}) is unique, even though the objective is linear and the feasible region need not be strictly convex. To see this, suppose that two distinct centers $\bm{c}^1\ne \bm{c}^2$ attain the same optimal radius $r^*$. Then, every input ball $\mathcal B_i$ is contained in both $\mathcal B(\bm{c}^1,r^*)$ and $\mathcal B(\bm{c}^2,r^*)$. Set $\bar{\bm{c}}:=(\bm{c}^1+\bm{c}^2)/2, \; \delta:=\|\bm{c}^1-\bm{c}^2\|/2>0.$
Then, for every point $\bm{x}\in\cup_{i=1}^m\mathcal B_i$, $\|\bm{x}-\bar{\bm{c}}\|^2+\delta^2
=\frac{1}{2}\|\bm{x}-\bm{c}^1\|^2 + \frac{1}{2}\|\bm{x}-\bm{c}^2\|^2
\leq (r^*)^2,$
and hence $\|\bm{x}-\bar{\bm{c}}\|\le \sqrt{(r^*)^2-\delta^2}<r^*$. Thus, all input balls are contained in a ball with center $\bar{\bm{c}}$ and radius strictly smaller than $r^*$, contradicting optimality. Therefore, the optimal pair $(r^*,\bm{c}^*)$ is unique.

Slater's condition also holds. Indeed, choose any $\bm{c}\in\mathbb{R}^d$ and $r>\max_{i\in[m]}\bigl(\|\bm{c}-\bm{c}_i\|+r_i\bigr).$
Then, for every $i\in[m]$,
\begin{equation*}
\left\|\bm{c}-\bm{c}_i \right\| < r-r_i
\quad \Longleftrightarrow \quad
\begin{pmatrix}
r-r_i \\
\bm{c}-\bm{c}_i
\end{pmatrix}
\in \operatorname{int}(\mathcal{Q}_{d+1}).
\end{equation*}
Consequently, strong duality holds for \eqref{eq:meb-socp}.

We next derive the dual problem of \eqref{eq:meb-socp}. Recall that $\mathcal Q_{d+1}$ is self-dual. We introduce the multiplier $(u_i,\bm{v}_i)\in\mathcal Q_{d+1}$ for each conic constraint, and write the Lagrangian function as follows:
\begin{equation}\label{eq-lag}
\mathcal L(r,\bm{c};\{u_i,\bm{v}_i\})
:= r
- \sum_{i=1}^m u_i(r-r_i)
- \sum_{i=1}^m \bm{v}_i^\top(\bm{c}-\bm{c}_i).
\end{equation}
Minimizing \eqref{eq-lag} over $(r,\bm{c})$ gives the dual feasibility conditions
$\sum_{i=1}^m u_i=1$ and $\sum_{i=1}^m\bm{v}_i=\bm{0}$. Hence, the dual problem is
\begin{equation}\label{eq:dual}
\max_{\{u_i,\,\bm{v}_i\}}
\quad
\sum_{i=1}^m \bigl(u_i r_i + \bm{v}_i^\top\bm{c}_i\bigr)
\qquad \text{s.t.} \qquad
\sum_{i=1}^m u_i = 1, \quad
\sum_{i=1}^m \bm{v}_i = \bm{0}, \quad
\begin{pmatrix}
u_i \\ \bm{v}_i
\end{pmatrix}
\in\mathcal{Q}_{d+1},
\quad \forall\,i\in[m].
\end{equation}
The corresponding Karush--Kuhn--Tucker (KKT) system is
\begin{equation*}
\bm{0} =
\begin{pmatrix}1 \\ \bm{0}\end{pmatrix}
- \sum_{i=1}^m\begin{pmatrix}
u_i \\
\bm{v}_i
\end{pmatrix},
\qquad
\bm{0} \in
\begin{pmatrix} u_i \\ \bm{v}_i \end{pmatrix}
+\mathcal{N}_{\mathcal{Q}_{d+1}}
\left(\begin{pmatrix}
r-r_i\\
\bm{c}-\bm{c}_i
\end{pmatrix}\right),
\quad \forall\,i\in[m],
\end{equation*}
where $\mathcal N_{\mathcal Q_{d+1}}$ denotes the normal cone of $\mathcal Q_{d+1}$. The second inclusion encodes primal feasibility, dual cone feasibility, and complementarity. By Slater's condition, the KKT system admits at least one solution. For convenience, we denote the above KKT system compactly by $\mathcal S(r,\bm{c},\{u_i,\bm{v}_i\}_{i=1}^m)=\bm{0},$
and write $\mathcal S^{-1}(\bm{0})$ for the set of KKT solutions.

Moreover, for a given penalty parameter $\sigma>0$, the augmented Lagrangian function for the primal problem \eqref{eq:meb-socp} can be expressed as (see, e.g.,~\cite[Example 11.57]{rockafellar1998variational})
\begin{equation*}
\phi_\sigma(r,\bm{c};\{u_i,\bm{v}_i\})
:= r + \frac{\sigma}{2}\sum_{i=1}^m\left\|\Pi_{\mathcal Q_{d+1}}
\left(\begin{pmatrix}
r_i-r \\ \bm{c}_i-\bm{c}
\end{pmatrix}
+
\frac{1}{\sigma}
\begin{pmatrix}
u_i \\ \bm{v}_i
\end{pmatrix}
\right)\right\|^2
- \frac{1}{2\sigma}\sum_{i=1}^m
\left\|\begin{pmatrix}
u_i \\ \bm{v}_i
\end{pmatrix}
\right\|^2.
\end{equation*}
With the multipliers $\{(u_i,\bm{v}_i)\}_{i=1}^m$ fixed, differentiating $\phi_\sigma$ with respect to $(r,\bm{c})$ gives
\begin{equation*}
\nabla \phi_\sigma(r,\bm{c})
=
\begin{pmatrix}
1 \\ \bm{0}
\end{pmatrix}
- \sigma\sum_{i=1}^m
\Pi_{\mathcal Q_{d+1}}
\left(
\begin{pmatrix}
r_i-r \\ \bm{c}_i-\bm{c}
\end{pmatrix}
+
\frac{1}{\sigma}
\begin{pmatrix}
u_i \\ \bm{v}_i
\end{pmatrix}
\right).
\end{equation*}

The projection onto $\mathcal Q_{d+1}$, which is the basic computational primitive in both $\phi_\sigma$ and $\nabla\phi_\sigma$, has the closed form
\begin{equation*}
\Pi_{\mathcal Q_{d+1}}
\left(
\begin{pmatrix}
a \\ \bm{b}
\end{pmatrix}\right)
=
\begin{cases}
\begin{pmatrix}
a \\ \bm{b}
\end{pmatrix}, & \|\bm{b}\|\leq a, \\
\bm{0}, & \|\bm{b}\|\leq -a, \\
\displaystyle
\frac{a+\|\bm{b}\|}{2}
\begin{pmatrix}
1 \\ \bm{b}/\|\bm{b}\|
\end{pmatrix},
& \text{otherwise},
\end{cases} \qquad \forall \begin{pmatrix}
a \\ \bm{b}
\end{pmatrix}\in \mathbb{R}^{d+1}.
\end{equation*}
Therefore, each evaluation of $\phi_\sigma$ and $\nabla\phi_\sigma$ consists of $m$ independent second-order cone projections. This separable structure is the main reason that the augmented Lagrangian framework can be implemented efficiently on GPUs.

The algorithmic objective is therefore to preserve these independent cone projections rather than pass the SOCP to a general solver, while retaining a principled path to high accuracy. The next section introduces a ripALM framework using the function $\phi_\sigma$ and specializes it so that the dominant constraint-dependent evaluations can be organized as independent per-ball computations followed by parallel reductions.

\section{A relative-type inexact proximal augmented Lagrangian method}\label{sec:ripalm-meb}

The proximal augmented-Lagrangian outer iteration, auxiliary-sequence update, relative stopping rule, and global convergence framework are inherited from general ripALM~\cite{zhu2024ripalm,yang2025convergence}. Our MEB specialization builds on this framework by deriving explicit gradient and generalized Hessian formulas using closed-form cone projections and their derivatives, and evaluating the per-ball contributions in parallel on GPUs.

For ease of presentation, we summarize the main notation used below. Subscripts $i$ identify input balls, superscripts $k$ identify outer ripALM iterations, and superscripts $j$ identify inner semismooth Newton iterations. The outer state consists of the primal pair $(r^k,\bm{c}^k)$, the conic multipliers $(u_i^k,\bm{v}_i^k)$, and the auxiliary variables $(\alpha^k,\bm{\beta}^k)$. The parameters $\sigma_k$, $\tau_k$, and $\rho$ denote the penalty, proximal weight, and relative inner-accuracy threshold, respectively.

We are now ready to present a relative-type inexact proximal augmented Lagrangian method (ripALM) for solving the SOCP reformulation \eqref{eq:meb-socp} in Algorithm \ref{alg:ipalm-meb}. One can see that, at the $k$th outer iteration, ripALM approximately minimizes the strongly convex proximal augmented Lagrangian function, updates the conic multipliers by the projection formula in \eqref{eq:meb-multiplier-update}, and accepts the inner approximate solution once a relative-type stopping condition is satisfied. This relative-type condition is important computationally: it avoids prescribing a summable sequence of absolute inner tolerances and instead measures the inner error against the actual progress made during the current outer iteration. Specifically, the vector $\bm{g}^{k+1}$ on the left-hand side is the gradient of the proximal augmented Lagrangian function evaluated at the accepted inner iterate, while the right-hand side measures the squared changes in the multipliers and primal variables. Thus, the inner problem is solved more accurately only when the outer iteration makes correspondingly smaller progress, as the outer iterates approach a KKT point.

\begin{algorithm}[htb!]
\caption{ripALM for solving the MEB reformulation \eqref{eq:meb-socp}}\label{alg:ipalm-meb}
\begin{algorithmic}[1]
\State \textbf{Input:}
$\rho\in[0,1)$,
$\{\sigma_k\}\subset\mathbb R_{++}$,
$\{\tau_k\}\subset\mathbb R_{++}$,
$(r^0,\bm{c}^0)\in\mathbb R\times\mathbb R^d$,
$(\alpha^0,\bm{\beta}^0)\in\mathbb R\times\mathbb R^d$, and
$(u_i^0,\bm{v}_i^0)\in\mathcal Q_{d+1}$, $\forall\,i\in[m]$.
\For{$k=0,1,2,\ldots$ until termination}

    \State Approximately solve the proximal augmented Lagrangian subproblem
    \begin{equation}\label{eq:meb-prox-alm-subproblem}
    \begin{aligned}
        \min_{r\in\mathbb R,\,\bm{c}\in\mathbb R^d}\;
        \Psi_k(r,\bm{c})
        &:=
        \phi_{\sigma_k}
        (r,\bm{c};\{u_i^k,\bm{v}_i^k\})
        +
        \frac{\tau_k}{2\sigma_k}
        \left\|
        \begin{pmatrix}
            r-r^k\\
            \bm{c}-\bm{c}^k
        \end{pmatrix}
        \right\|^2,
    \end{aligned}
    \end{equation}

    to find $(r^{k+1},\bm{c}^{k+1})$ satisfying the relative-type stopping condition
    \begin{equation}\label{eq:meb-relative-error-condition}
    2\sigma_k
    \left|
    \left\langle
    \begin{pmatrix}
        \alpha^k-r^{k+1}\\
        \bm{\beta}^k-\bm{c}^{k+1}
    \end{pmatrix},
    \,\bm{g}^{k+1}
    \right\rangle
    \right|
    +
    \sigma_k^2\|\bm{g}^{k+1}\|^2
    \leq
    \rho
    \left[
    \sum_{i=1}^m
    \left\|
    \begin{pmatrix}
        u_i^{k+1}-u_i^k\\
        \bm{v}_i^{k+1}-\bm{v}_i^k
    \end{pmatrix}
    \right\|^2
    +
    \tau_k
    \left\|
    \begin{pmatrix}
        r^{k+1}-r^k\\
        \bm{c}^{k+1}-\bm{c}^k
    \end{pmatrix}
    \right\|^2
    \right],
    \end{equation}

    where
    \begin{equation}
    \label{eq:meb-multiplier-update}
        \begin{pmatrix}
            u_i^{k+1} \\ \bm{v}_i^{k+1}
        \end{pmatrix}
        :=
        \sigma_k
        \Pi_{\mathcal Q_{d+1}}
        \left(
            \begin{pmatrix}
                r_i-r^{k+1}\\
                \bm{c}_i-\bm{c}^{k+1}
            \end{pmatrix}
            +
            \frac{1}{\sigma_k}
            \begin{pmatrix}
                u_i^k \\ \bm{v}_i^k
            \end{pmatrix}
        \right),
        \qquad \forall\,i\in[m],
    \end{equation}

    and
    \begin{equation*}
        \bm{g}^{k+1}
        := \nabla\Psi_k(r^{k+1}, \bm{c}^{k+1}) =
        \begin{pmatrix}
            1 \\ \bm{0}
        \end{pmatrix}
        -
        \sum_{i=1}^m
        \begin{pmatrix}
            u_i^{k+1} \\ \bm{v}_i^{k+1}
        \end{pmatrix}
        +
        \frac{\tau_k}{\sigma_k}
        \begin{pmatrix}
            r^{k+1}-r^k \\ \bm{c}^{k+1}-\bm{c}^k
        \end{pmatrix}.
    \end{equation*}

    \State Update the auxiliary sequence
    \begin{equation*}
        \begin{pmatrix}
            \alpha^{k+1} \\ \bm{\beta}^{k+1}
        \end{pmatrix}
        =
        \begin{pmatrix}
            \alpha^k \\ \bm{\beta}^k
        \end{pmatrix}
        -
        \sigma_k\bm{g}^{k+1}.
    \end{equation*}

\EndFor
\end{algorithmic}
\end{algorithm}

We next present the convergence properties for ripALM. The results follow from the general ripALM theory in~\cite{zhu2024ripalm,yang2025convergence,zhu2026d}; the only additional observation used here is the uniqueness of the primal solution of problem \eqref{eq:meb-socp}.

\begin{theorem}\label{thm-convergence}
Let $\rho\in[0,1)$. Suppose that $\mathcal S^{-1}(\bm{0})$ is nonempty, $\{\sigma_k\}\subset\mathbb R_{++}$ satisfies $\sigma_k\ge \sigma_{\rm min}>0$ for all $k\ge0$, and $\{\tau_k\}\subset\mathbb R_{++}$ satisfies $\tau_k\ge \tau_{\rm min}>0,\;
\tau_{k+1}\le (1+\nu_k)\tau_k,\;
\nu_k\ge0,\;
\sum_{k=0}^{\infty}\nu_k<\infty.$
Let $\{(r^k,\bm{c}^k)\}$ and $\{(u_i^k,\bm{v}_i^k):i\in[m]\}$ be generated by the ripALM in Algorithm~\ref{alg:ipalm-meb}. Then, these sequences are bounded. Moreover, $\{(r^k,\bm{c}^k)\}$ converges to the unique optimal solution of problem \eqref{eq:meb-socp}, and $\{(u_i^k,\bm{v}_i^k):i\in[m]\}$ converges to an optimal solution of the dual problem \eqref{eq:dual}.
\end{theorem}
\begin{proof}
The boundedness of the generated sequences and the convergence of the dual sequence follow the same proof as \cite[Theorem 3.1]{yang2025convergence}. The same proof also shows that every accumulation point of $\{(r^k,\bm{c}^k)\}$ is an optimal solution of \eqref{eq:meb-socp}. Since \eqref{eq:meb-socp} has the unique optimal solution $(r^*,\bm{c}^*)$, every accumulation point of $\{(r^k,\bm{c}^k)\}$ must be $(r^*,\bm{c}^*)$. If the whole sequence did not converge to $(r^*,\bm{c}^*)$, there would exist an $\varepsilon>0$ and a subsequence $\{(r^{k_j},\bm{c}^{k_j})\}$ such that
\begin{equation*}
\big\|(r^{k_j},\bm{c}^{k_j}) - (r^*, \bm{c}^*) \big\| \geq \varepsilon,
\quad~~\text{for all}~~j\geq0.
\end{equation*}
By boundedness, this subsequence has a further convergent subsequence with limit $(\bar r,\bar{\bm{c}})$. The limit is an accumulation point and hence is optimal. Uniqueness gives $(\bar r,\bar{\bm{c}})=(r^*,\bm{c}^*)$, contradicting the displayed inequality. Therefore, the whole primal sequence converges to $(r^*,\bm{c}^*)$.
\end{proof}

\begin{theorem}\label{thm-complexity}
Suppose that all assumptions of Theorem~\ref{thm-convergence} hold. Define the weighted ergodic average sequence
\begin{equation*}
\begin{pmatrix}
\hat{r}^k \\ \hat{\bm{c}}^k
\end{pmatrix} := \frac{1}{\sum_{\ell=0}^{k-1}\sigma_\ell}\sum_{\ell=0}^{k-1}\sigma_\ell \begin{pmatrix}
r^{\ell+1} \\ \bm{c}^{\ell+1}
\end{pmatrix},\quad \forall\,k\geq 1.
\end{equation*}
Then, it holds that
\begin{equation*}
\max\left\{\left\lvert \hat{r}^k - r^*\right\rvert, \,\max_{i\in[m]} \left\{\mathrm{dist}\left(\begin{pmatrix}
\hat{r}^k - r_i \\ \hat{\bm{c}}^k - \bm{c}_i
\end{pmatrix}, \mathcal{Q}_{d+1} \right) \right\}\right\} \leq \mathcal{O}\left(\frac{1}{\sum_{\ell=0}^{k-1}\sigma_\ell}\right),\quad \forall\,k\geq 1.
\end{equation*}
\end{theorem}
\begin{proof}
This is the specialization of the global ergodic rate for ripALM in~\cite[Theorem 3.3]{yang2025convergence} to the conic formulation \eqref{eq:meb-socp}. Slater's condition ensures strong duality and boundedness of the multiplier set, while the uniqueness of the primal solution identifies the limiting objective value with $r^*$. Applying the general bound to the objective residual and the conic feasibility residual gives the stated estimate.
\end{proof}

Theorem~\ref{thm-complexity} shows that the global ergodic rate of ripALM is determined by the growth of the penalty parameter $\sigma_k$. In particular, a fixed penalty parameter $\sigma_k=\sigma>0$ yields an $\mathcal{O}(1/k)$ rate, while a linearly increasing choice $\sigma_{k}=\sigma_{0}(k+1)$ with $\sigma_{0}>0$ improves the rate to $\mathcal{O}(1/k^2)$. Moreover, if $\sigma_{k}=\sigma_{0}c^{k}$ for constants $\sigma_{0}>0$ and $c>1$, the corresponding rate becomes $\mathcal{O}(c^{-k})$. These estimates suggest that increasing $\sigma_k$ can accelerate the outer convergence. In practice, however, an overly aggressive increase may lead to ill-conditioned inner subproblems and numerical instability. Thus, the update of the penalty parameter should balance faster outer progress against the conditioning and reliable solution of the inner subproblems.

\section{Efficient GPU implementation}\label{sec:gpu-implementation}

In this section, we discuss how the ripALM framework in Section~\ref{sec:ripalm-meb} can be implemented efficiently on parallel architectures, with particular emphasis on GPU acceleration. The key observation is that the dominant operations in the evaluation of the proximal augmented Lagrangian subproblem objective, its gradient, its generalized Hessian, and the multiplier update are separable across the input balls. Thus, the expensive part of each inner iteration can be organized as a parallel map over the constraints, followed by a small number of global reductions.

To simplify notation, we suppress the outer-iteration index in this section. Let $(\bar{r},\bar{\bm{c}})$ denote the current proximal center, let $\{(u_i,\bm{v}_i)\}_{i=1}^m$ denote the current multipliers, and let $\sigma,\tau>0$ denote the current penalty and proximal parameters. The basic conic primitive is the batched projection
\begin{equation*}
\begin{pmatrix}
u_i^+ \\
\bm{v}_i^+
\end{pmatrix}
=
\sigma\Pi_{\mathcal{Q}_{d+1}}
\left(
\begin{pmatrix}
r_i-r \\
\bm{c}_i-\bm{c}
\end{pmatrix}
+
\frac{1}{\sigma}
\begin{pmatrix}
u_i \\
\bm{v}_i
\end{pmatrix}
\right), \quad \forall\,i\in [m].
\end{equation*}
These $m$ projections are mutually independent. Denoting the current proximal augmented Lagrangian subproblem objective by $\Psi$, one evaluation of its gradient and generalized Hessian matrix (used in the subsequent semismooth Newton step) takes the form
\begin{equation*}
\nabla \Psi(r,\bm{c})
=
\begin{pmatrix}
1 \\
\bm{0}
\end{pmatrix}
-
\sum_{i=1}^m
\begin{pmatrix}
u_i^+\\ \bm v_i^+
\end{pmatrix}
+
\frac{\tau}{\sigma}
\begin{pmatrix}
r-\bar r \\
\bm c-\bar{\bm c}
\end{pmatrix},\quad
H := \frac{\tau}{\sigma}I_{d+1} + \sigma\sum_{i = 1}^m P_i
\in\widehat\nabla^2\Psi(r,\bm{c}),
\end{equation*}
where $P_i\in \partial \Pi_{\mathcal{Q}_{d+1}}(z_i)$ is an element of the Clarke generalized Jacobian~\cite{clarke1990optimization} of the projection mapping at $z_i :=
\begin{pmatrix}
r_i-r \\
\bm{c}_i-\bm{c}
\end{pmatrix}
+
\frac{1}{\sigma}
\begin{pmatrix}
u_i \\
\bm{v}_i
\end{pmatrix},$
and $\widehat{\nabla}^2\Psi(r,\bm{c})$ denotes the set of matrices $H$ obtained from these choices of $P_{i}$, which serves as a surrogate generalized Hessian at $(r,\bm{c})$. Since the second-order cone projection is firmly nonexpansive, the matrices $P_i$ are symmetric positive semidefinite~\cite{liang2021inexact}. Consequently, the matrix $H$ is symmetric positive definite because of the term $(\tau/\sigma)I_{d+1}$. This map-reduce pattern is well aligned with GPU execution: each cone projection uses only the data associated with one ball, while the sums of the projected multipliers and generalized Jacobian contributions are standard parallel reductions.

A structure-of-arrays layout is preferable for high-throughput implementations. Specifically, we store
\begin{equation*}
R:= (r_1,\ldots,r_m)^{\top}\in\mathbb R^m,\quad
C:=(\bm c_1^{\top};\ldots;\bm c_m^{\top})\in\mathbb R^{m\times d},\quad
U:= (u_1,\ldots,u_m)^{\top}\in\mathbb R^m,\quad
V:=(\bm v_1^{\top};\ldots;\bm v_m^{\top})\in\mathbb R^{m\times d}.
\end{equation*}
Under this layout, the second-order cone projections can be evaluated using predicated operations or elementwise masks, avoiding serial branching over the index $i$. The projected multipliers also provide the per-ball contributions to the gradient. When an explicit generalized Hessian matrix is used, the local matrices $P_i$ are accumulated through reductions; when $d$ is large, it is preferable to avoid forming these matrices and instead compute Hessian-vector products directly. Reusing intermediate quantities and limiting temporary storage can reduce memory traffic.

Our implementation uses PyTorch tensors: $R,U\in\mathbb R^m$, $C,V\in\mathbb R^{m\times d}$, the projection work arrays in $\mathbb R^{m\times(d+1)}$, the gradient in $\mathbb R^{d+1}$, and an explicit Hessian in $\mathbb R^{(d+1)\times(d+1)}$. Batched norms and masks implement the cone projection; \texttt{sum} and matrix multiplication perform the reductions and Hessian accumulation; and the Newton system is solved using Cholesky factorization and triangular solves. Armijo trials are serial, but each trial evaluation is batched over all balls in the current problem. The constraint-wise arrays require $\mathcal O(md)$ memory, the Hessian requires $\mathcal O(d^2)$ memory, its accumulation costs $\mathcal O(md^2)$ in the current dense implementation, and Cholesky factorization costs $\mathcal O(d^3)$. GPU execution is handled by PyTorch and its underlying default CUDA kernels and libraries; we use no custom CUDA kernels or manual thread/block configuration. Host decisions in the line search introduce scalar synchronization.

Figure~\ref{fig:gpu-workflow} shows the overall mixed-precision workflow and the GPU loop shared by both ripALM phases. The first panel shows the FP32 coarse solve, constraint screening, FP64 refinement, and full-set feasibility check, with constraint reinsertion when needed. The second panel shows the inner semismooth Newton iterations and outer ripALM updates, including batched evaluations, reductions, and stopping tests.

\begin{figure}[!]
\centering
\begin{tikzpicture}[
  font=\small,
  box/.style={draw=black, rounded corners=1.5pt, fill=black!3, text=black,
    align=center, inner sep=3pt, minimum height=9.5mm, text width=39mm},
  smallbox/.style={box, text width=22mm, minimum height=8.5mm},
  decision/.style={diamond, draw=black, fill=black!3, text=black,
    align=center, aspect=2.25, inner sep=1.5pt},
  flow/.style={-{Latex[length=2.1mm]}, thick, black},
  loop/.style={-{Latex[length=2.1mm]}, thick, black, rounded corners=2pt}
]
\node[text=black,font=\bfseries] at (1.0,0.85) {(a) Mixed-precision workflow};
\node[box] (data) at (0,0) {FP64 master data\\$R,C$ retained unchanged};
\node[box] (low) at (0,-1.55) {Phase I (FP32)\\coarse ripALM solve};
\node[box] (screen) at (0,-3.10) {boundary screening\\form $\mathcal I_{\rm red}$};
\node[box] (high) at (0,-4.65) {Phase II (FP64)\\warm-started solve on $\mathcal I_{\rm red}$};
\node[decision] (check) at (0,-6.25) {all $m$ constraints\\within tolerance?};
\node[box] (accept) at (0,-7.85) {accepted candidate\\$(\bar r,\bar{\bm c})$};
\node[smallbox] (add) at (3.25,-6.25) {add violated\\indices};
\draw[flow] (data) -- node[right,text=black]{cast copy} (low);
\draw[flow] (low) -- (screen);
\draw[flow] (screen) -- (high);
\draw[flow] (high) -- (check);
\draw[flow] (check) -- node[right,text=black]{yes} (accept);
\draw[flow] (check) -- node[above,pos=0,text=black]{no} (add);
\draw[loop] (add.north) |- (high.east);

\draw[black!65,densely dashed] (4.75,0.55) -- (4.75,-9.55);

\node[text=black,font=\bfseries] at (8.25,0.85) {(b) GPU ripALM loop used in both phases};
\node[box] (state) at (8.25,0) {phase data and constraint set\\primal and dual variables\\auxiliary variables};
\node[box] (map) at (8.25,-1.50) {MAP over current constraints\\SOC projections and local $\bm g,H$};
\node[box] (reduce) at (8.25,-2.95) {REDUCE\\assemble global $\bm g,H$};
\node[decision,draw=black] (relative) at (8.25,-4.50) {relative inner\\test met?};
\node[box,draw=black] (step) at (8.25,-6.05) {dense semismooth Newton solve\\and batched Armijo update};
\node[box] (outer) at (8.25,-7.55) {accept inner iterate\\update multipliers and\\auxiliaries; check KKT};
\node[box] (candidate) at (8.25,-9.05) {return phase candidate};
\draw[flow] (state) -- (map);
\draw[flow] (map) -- (reduce);
\draw[flow,black] (reduce) -- (relative);
\draw[flow,black] (relative) -- node[right]{no} (step);
\draw[loop,black] (relative.east) -- node[above]{yes} (11.20,-4.50) |- (outer.east);
\draw[loop,black] (step.west) -- (5.75,-6.05) |- (map.west);
\draw[flow] (outer) -- node[right,text=black]{converged} (candidate);
\draw[loop] (outer.west) -- node[above,pos=0.7,text=black]{continue} (5.10,-7.55) |- (map.west);
\end{tikzpicture}
\caption{Mixed-precision workflow and GPU ripALM loop.}\label{fig:gpu-workflow}
\Description{Two side-by-side flow diagrams show the overall mixed-precision workflow and the GPU ripALM loop. The workflow proceeds from unchanged FP64 master data through an FP32 coarse solve, boundary screening, an FP64 reduced solve, and a full-set feasibility check with a prescribed tolerance. Constraints violated beyond the tolerance are added to the reduced set. The GPU loop maps cone projections over the current constraints, reduces local gradient and Hessian terms, and tests relative inner accuracy. If the test fails, a semismooth Newton and Armijo step is taken before returning to the map and reduction. Otherwise, the inner iterate is accepted, the multipliers and auxiliary variables are updated, and KKT convergence is checked.}
\end{figure}

\subsection{Semismooth Newton method for the inner subproblem}\label{subsec:ssn-inner}

A first-order method, such as the accelerated gradient method~\cite{beck2017first}, could solve the subproblem \eqref{eq:meb-prox-alm-subproblem}. Near a KKT point, however, the relative condition \eqref{eq:meb-relative-error-condition} requires a much smaller inner residual, and every additional inner iteration entails another pass over the constraints in the current problem. We therefore use semismooth Newton \cite{qi1993nonsmooth} to exploit its fast local convergence when high accuracy is required. Given the current inner iterate $(r^j,\bm c^j)$, the Newton direction $(\Delta r^j,\Delta \bm c^j)$ is computed from
\begin{equation*}
H^j \begin{pmatrix}
\Delta r^j \\
\Delta \bm{c}^j
\end{pmatrix}
= -\nabla\Psi(r^j,\bm{c}^j),
\qquad
H^j\in \widehat{\nabla}^2\Psi(r^j,\bm{c}^j).
\end{equation*}
The next iterate is then updated as $(r^{j+1},\bm c^{j+1}) = (r^j,\bm c^j) + \eta_j(\Delta r^j,\Delta \bm c^j),$
where the step size $\eta_j\in (0,1]$ is obtained by a standard Armijo line search for $\Psi$. The proximal term makes the subproblem strongly convex, while the second-order cone projection is strongly semismooth~\cite{chen2003complementarity}. Consequently, after the iterates enter a neighborhood of the subproblem solution, the semismooth Newton method typically accepts full steps (namely, $\eta_j=1$) and reduces the inner residual rapidly; see~\cite{zhao2010newton} for more details on a related semismooth Newton based augmented Lagrangian framework.

The fast local convergence of semismooth Newton motivates its use when stringent inner tolerances are required. We have not performed a direct comparison with first-order inner solvers. In our implementation, we explicitly assemble $H^j$ and solve the $(d+1)\times(d+1)$ system by Cholesky factorization and triangular solves. All reported experiments use this dense solver, including the high-dimensional and real-data cases. For larger $d$, a matrix-free Krylov method could use the Hessian--vector product $H^{j}\bm{q}=\frac{\tau}{\sigma}\bm{q}+\sigma\sum_{i=1}^{m}P_{i}^{j}\bm{q}$ for any $\bm{q}\in\mathbb{R}^{d+1}$. This product can be evaluated through constraint-wise operations and parallel reductions without forming the dense matrix $H^{j}$. The matrix-free variant and comparisons with first-order inner solvers are left for future work.

Algorithm~\ref{alg:ssn-subproblem} summarizes the semismooth Newton method for solving the subproblem \eqref{eq:meb-prox-alm-subproblem}.

\begin{algorithm}[htb!]
\caption{Semismooth Newton method for solving the subproblem \eqref{eq:meb-prox-alm-subproblem}}\label{alg:ssn-subproblem}
\begin{algorithmic}[1]
\State \textbf{Input:}
proximal center $(\bar{r},\bar{\bm{c}})$, auxiliary point $(\alpha,\bm{\beta})$, multipliers $\{(u_i,\bm v_i)\}_{i=1}^m$, parameters $\sigma,\tau>0$, and initial inner iterate $(r^0,\bm c^0)$.

\For{$j=0,1,2,\ldots$}
    \State In parallel over $i=1,\ldots,m$, compute
    \begin{equation*}
        z_i(r^j,\bm c^j)
        =
        \begin{pmatrix}
            r_i-r^j\\ \bm c_i-\bm c^j
        \end{pmatrix}
        +
        \frac{1}{\sigma}
        \begin{pmatrix}
            u_i\\ \bm v_i
        \end{pmatrix}, \quad
        \begin{pmatrix}
            u_i^+\\ \bm v_i^+
        \end{pmatrix}
        =
        \sigma\Pi_{\mathcal Q_{d+1}}(z_i(r^j,\bm c^j)),
        \quad
        P_i^j\in \partial \Pi_{\mathcal Q_{d+1}}(z_i(r^j,\bm c^j)).
    \end{equation*}
    \State Perform reductions
    \begin{equation*}
        \bm{g}^j
        =
        \begin{pmatrix}
            1\\\bm 0
        \end{pmatrix}
        -
        \sum_{i=1}^m
        \begin{pmatrix}
            u_i^+\\ \bm v_i^+
        \end{pmatrix}
        +
        \frac{\tau}{\sigma}
        \begin{pmatrix}
            r^j-\bar r\\ \bm c^j-\bar{\bm c}
        \end{pmatrix},\quad
        H^j
        =
        \frac{\tau}{\sigma}I_{d+1}
        +
        \sigma
        \sum_{i=1}^m P_i^j.
    \end{equation*}
    \If{the relative-type stopping condition \eqref{eq:meb-relative-error-condition} is satisfied by the tentative quantities}
        \State \textbf{return} $(r^j,\bm c^j)$ and the corresponding multipliers
        $\{(u_i^+,\bm v_i^+)\}_{i=1}^m$.
    \EndIf

    \State Compute the Newton direction $(\Delta r^j, \Delta \bm c^j)$ from $H^j\begin{pmatrix}
            \Delta r^j \\ \Delta \bm c^j
        \end{pmatrix}=-\bm{g}^j$.
    \State Set $(r^{j+1},\bm c^{j+1})=(r^j,\bm c^j)+\eta_j(\Delta r^j,\Delta \bm c^j)$, where $\eta_j\in (0,1]$ is computed by Armijo line search.
\EndFor
\end{algorithmic}
\end{algorithm}

In the stopping test in Algorithm~\ref{alg:ssn-subproblem}, the current outer data in \eqref{eq:meb-relative-error-condition} are represented by $(\bar r,\bar{\bm c})$, $(\alpha,\bm{\beta})$, $\sigma$, $\tau$, and $\{(u_i,\bm v_i)\}_{i=1}^m$; the tentative accepted point is $(r^j,\bm c^j)$, with gradient $\bm{g}^j$ and multipliers $\{(u_i^+,\bm v_i^+)\}_{i=1}^m$. In the numerical implementation, the full semismooth Newton step (namely, $\eta_j=1$) is typically accepted after a few line-search iterations, and the inner residual then decreases rapidly. This behavior complements the relative-type stopping rule in Algorithm~\ref{alg:ipalm-meb}: the early outer iterations do not require oversolving the subproblem, while later outer iterations benefit from the fast local convergence of the semismooth Newton method when the prescribed KKT tolerance is stringent.

\subsection{Mixed-precision reduction strategy}\label{sec:mixed-precision-reduction}

We next describe a mixed-precision reduction strategy used to accelerate the application of ripALM to the MEB reformulation \eqref{eq:meb-socp}. The main idea is to run ripALM first in low precision and with a loose stopping tolerance, thereby obtaining approximate geometric information about the boundary of the enclosing ball at low computational cost. This information is then used to select a smaller set of constraints that are likely to be active or nearly active. We then apply ripALM in high precision to the reduced problem until the prescribed KKT tolerance is reached. Thus, low precision is used only for screening and warm starting; the candidate ball returned from the reduced problem is always checked against the full set of constraints.

After computing a coarse low-precision solution $(\widehat{r},\,\widehat{\bm{c}},\,\{(\widehat{u}_{i},\,\widehat{\bm{v}}_{i})\}_{i=1}^{m})$, we define $\widehat{q}_{i}$ and select the retained set using a screening ratio $\gamma_{\mathrm{scr}}\in(0,1]$: $\widehat q_i:=r_i+\|\widehat{\bm c}-\bm c_i\|,
\;
\mathcal{I}_{\rm red}
:= \left\{i\in[m]:\widehat q_i\geq \gamma_{\rm scr}\widehat r\right\}.
$
A smaller $\gamma_{\rm scr}$ gives a more conservative reduction. After constructing $\mathcal I_{\rm red}$, we apply ripALM to the reduced problem
\begin{equation*}
\min_{r\in\mathbb{R},\; \bm{c}\in\mathbb{R}^d}
\quad r
\qquad
\mathrm{s.t.}\qquad
r-r_i \geq \|\bm{c}-\bm{c}_i\|,
\quad \forall\,i\in \mathcal I_{\rm red},
\end{equation*}
with the prescribed final stopping tolerance. This run of ripALM is initialized by the coarse low-precision solution:
\begin{equation*}
(r^0,\bm{c}^0)=(\widehat{r},\widehat{\bm{c}}),
\qquad
(u_i^0,\bm{v}_i^0)=(\widehat{u}_i,\widehat{\bm{v}}_i),
\quad i\in\mathcal I_{\rm red}.
\end{equation*}
The reduction has a simple exact certificate. Let $p_{\rm red}$ and $p_*$ be the reduced and full optimal values. Since the reduced problem removes constraints, $p_{\rm red}\leq p_*$. If an exact reduced optimizer is feasible for every original constraint, then $p_*\leq p_{\rm red}$; hence it is also a full optimizer.

For robustness, we further perform an a posteriori feasibility check against all original constraints after applying ripALM to the reduced problem. Let $(\bar{r},\bar{\bm{c}})$ denote the resulting radius and center. We fix an absolute full-set feasibility tolerance $\varepsilon_{\rm feas}>0$, which bounds the permitted boundary violation. All experiments use $\varepsilon_{\rm feas}=10^{-8}$. We then compute the maximal signed violation $\Delta_{\max} :=\max_{i\in[m]}\left\{r_i+\|\bar{\bm{c}}-\bm{c}_i\|-\bar{r}\right\}.$
The candidate is accepted only when $\Delta_{\max}\leq\varepsilon_{\rm feas}$. In exact arithmetic, the corrected ball $(\bar{r}+\max\{\Delta_{\max},0\},\bar{\bm{c}})$ is feasible for every original constraint.

Objective accuracy can be certified directly by duality, without an error-bound assumption. Let $\{(u_i,\bm v_i)\}_{i\in\mathcal I_{\rm red}}$ be any dual-feasible point for the reduced SOCP and define its objective value by
$
d_{\rm red}:=\sum_{i\in\mathcal I_{\rm red}}
\bigl(u_i r_i+\bm v_i^\top\bm c_i\bigr).
$
Extending these multipliers by zero on the discarded indices gives a dual-feasible point for the full problem. For an accepted candidate, set $r_{\rm corr}:=\bar r+\max\{\Delta_{\max},0\}$. In exact arithmetic, the correction makes $(r_{\rm corr},\bar{\bm c})$ feasible for the full problem, and weak duality yields
$
d_{\rm red}\leq p_{\rm red}\leq p_*\leq r_{\rm corr},
\;
0\leq r_{\rm corr}-p_*\leq r_{\rm corr}-d_{\rm red}.
$
Thus, $r_{\rm corr}-d_{\rm red}$ is an a posteriori upper bound on the objective error whenever a certified dual-feasible point is available. The KKT residual used for termination is reported separately and is not itself interpreted as an objective-error bound. The parameter $\gamma_{\rm scr}$ controls the amount of screening, whereas $\varepsilon_{\rm feas}$ controls acceptance by the full-set feasibility check. If $\Delta_{\max}>\varepsilon_{\rm feas}$, the indices violated beyond the tolerance $\mathcal{V} :=
\left\{i\in[m]: r_i+\|\bar{\bm{c}}-\bm{c}_i\|-\bar{r}>\varepsilon_{\rm feas}
\right\}$
are added to $\mathcal I_{\rm red}$, and ripALM is applied again to the updated reduced problem. This active-set refinement protects the method against overly aggressive pruning in the low-precision phase: any discarded constraint that is violated beyond the prescribed tolerance by the candidate ball is detected and reintroduced. Therefore, the mixed-precision reduction serves only as an acceleration mechanism, while consistency with the original problem is maintained by the full verification and active-set update steps.

\section{Numerical Experiments}\label{sec:numerical-experiments}

In this section, we evaluate the efficiency, accuracy, and scalability of the proposed ripALM framework for solving the MEB reformulation \eqref{eq:meb-socp}. The experiments compare ripALM with both specialized computational-geometry software and general-purpose conic optimization solvers. We also test the mixed-precision implementation of ripALM, as described in Section~\ref{sec:mixed-precision-reduction}, which first screens the constraints in low precision and then applies ripALM to the reduced problem with the prescribed final stopping tolerance.

{\bf Test instances.}
We follow the deterministic pseudo-random procedure used in~\cite{zhou2005efficient} to generate the test instances. Specifically, the centers and radii are generated from the pseudo-random sequence
\begin{equation*}
\psi_0 = 7,
\qquad
\psi_{j+1} = (445\psi_j+1) \bmod 4096,
\qquad j=0,1,2,\ldots,\qquad \bar\psi_j = \frac{\psi_j}{40.96},
\qquad j=1,2,\ldots.
\end{equation*}
The radii and center coordinates are then assigned successively from $\{\bar{\psi}_j\}$ in the order
\begin{equation*}
r_1,\ c_1(1),\ c_1(2),\ldots,c_1(d),\
r_2,\ c_2(1),\ldots,c_2(d),
\ \ldots,\
r_m,\ c_m(1),\ldots,c_m(d),
\end{equation*}
where $c_{i}(j)$ denotes the $j$th coordinate of $\bm{c}_{i}$ for $j=1,\ldots,d$. For the parameter sensitivity study, we additionally generate three independent instances with $(d,m)=(50,10^5)$. In each instance, every radius $r_i$ and center coordinate $c_i(j)$ is sampled independently from the uniform distribution $\mathcal U[0,100)$, using three random seeds.

For the real-data instances, we use 60,000 Fashion-MNIST images as points in $\mathbb R^{784}$~\cite{xiao2017fashion} and $10^6$ standardized HIGGS records as points in $\mathbb R^{28}$~\cite{baldi2014higgs}. Grouping all 814,277 FEMNIST images by writer and character~\cite{caldas2018leaf,flwrlabs2024femnist} yields 76,866 balls in $\mathbb R^{784}$ whose centers are the group means and whose radii are the maximum distances from group members to the corresponding group means. Thus, the nonzero FEMNIST radii quantify observed within-group variation. For the boundary-dense stress test, we construct an instance consisting of $10^{5}$ balls in $\mathbb{R}^{50}$, each with radius $0.1$, whose centers lie on the unit sphere and occur in opposite pairs; their exact MEB has center zero and radius $1.1$, so every ball is active.

{\bf Baseline methods.}
We compare ripALM and its mixed-precision implementation with three representative baselines. The first baseline is CGAL~\cite{cgal:eb-26a}. For the ball instances considered here, we develop a Python interface to CGAL's C++ implementation for the smallest enclosing sphere of spheres in fixed dimension, which is a natural specialized benchmark for \eqref{eq:meb}. The routine is run with its default parameter settings. The other two baselines are obtained through CVXPY\footnote{Available at \url{https://www.cvxpy.org/}} by solving the SOCP formulation \eqref{eq:meb-socp}. Specifically, we use SCS, a first-order splitting method for conic optimization~\cite{ocpb:16}, and CLARABEL, a primal-dual interior-point solver for conic optimization~\cite{Clarabel_2024}. For SCS, we set \texttt{eps}=$10^{-8}$ and \texttt{max\_iters}=20000. For CLARABEL, we set \texttt{tol\_gap\_abs}=\texttt{tol\_feas}=$10^{-8}$ and otherwise use its defaults.

The direct high-precision implementation of Algorithm~\ref{alg:ipalm-meb} is denoted by ripALM. Its mixed-precision variant, mp-ripALM, uses a coarse low-precision solve for screening and warm starting, followed by high-precision refinement on the reduced problem and a full-set feasibility check. Termination uses the following maximum KKT residual:
\begin{equation*}
\begin{aligned}
\mathrm{KKT}_{\max}(r,\bm{c},\{(u_i,\bm{v}_i)\}_{i=1}^m)
:=\max\Bigg\{
&\left|1-\sum_{i=1}^m u_i\right|,
\left\|\sum_{i=1}^m \bm{v}_i\right\|, \max_{1\le i\le m}
\operatorname{dist}\!\left(\begin{pmatrix}
r - r_i \\ \bm{c} - \bm{c}_i
\end{pmatrix},\mathcal{Q}_{d+1}\right),
\max_{1\le i\le m}
\operatorname{dist}\!\left(\begin{pmatrix}
u_i \\ \bm{v}_i
\end{pmatrix},\mathcal{Q}_{d+1}\right), \\
&\max_{1\le i\le m}
\left|(r - r_i)u_i + \langle \bm{c} - \bm{c}_i, \bm{v}_i\rangle\right|
\Bigg\},
\end{aligned}
\end{equation*}
where $\operatorname{dist}(\cdot,\mathcal{Q}_{d+1})$ denotes the Euclidean distance to the second-order cone $\mathcal{Q}_{d+1}$. The master data are retained in \texttt{float64}. Table~\ref{tab:solver-parameters} lists the default solver settings, which are used for all instances except in Section~\ref{sec:sensitivity}. The mixed-precision workflow uses $\gamma_{\rm scr}=0.98$, retains at least one constraint, checks full-set feasibility to $10^{-8}$, and performs at most two refinement rounds. The loose low-precision tolerance affects only screening and warm starting; final accuracy is controlled by high-precision refinement and the full-set feasibility check. As practical tuning guidance, increasing the penalty growth factor may improve outer convergence but worsen the conditioning of inner subproblems. A smaller screening ratio retains more constraints, whereas overly aggressive screening can require additional refinement.

\begin{table}[!t]
\caption{Default solver parameters (inner limits per subproblem; line-search limits per Newton step).}\label{tab:solver-parameters}
\centering
\footnotesize
\setlength{\tabcolsep}{4pt}
\renewcommand{\arraystretch}{1.08}

\begin{tabular*}{0.90\textwidth}{@{\extracolsep{\fill}}lccc@{}}
\toprule
Parameter & ripALM & mp-ripALM (low) & mp-ripALM (high) \\
\midrule
Arithmetic & \texttt{float64} & \texttt{float32} & \texttt{float64} \\
$\sigma_{0}$ / growth factor & $1$ / $2$ & $1$ / $1$ & $1$ / $2$ \\
$\tau$ / $\rho$ & $10^{-3}$ / $0.9$ & $10^{-3}$ / $0.9$ & $10^{-3}$ / $0.9$ \\
\addlinespace[2pt]
KKT tolerance & $10^{-8}$ & $10^{-2}$ & $10^{-8}$ \\
\addlinespace[2pt]
Maximum outer / inner iterations & $100$ / $50$ & $100$ / $10$ & $100$ / $20$ \\
Maximum line-search iterations & $10$ & $10$ & $10$ \\
\bottomrule
\end{tabular*}

\end{table}

{\bf Computing environment and performance metrics.}
All experiments are conducted on a workstation equipped with an Intel Core i7-13700K CPU, 64 GB of RAM, and an NVIDIA RTX 4070 GPU with 12 GB of memory. The GPU implementations of ripALM and mp-ripALM use PyTorch. The baseline software versions are CGAL 6.2-beta1, SCS 3.2.8, and CLARABEL 0.11.1; SCS and CLARABEL are run through CVXPY 1.8.2 on the CPU, and CGAL is also run on the CPU. To assess GPU acceleration, Section~\ref{sec:implementation-diagnostics} compares the same PyTorch implementation on the CPU and GPU using identical settings. The comparison between ripALM and mp-ripALM on the same GPU evaluates the effect of mixed-precision reduction. We compare the proposed methods with these baselines in terms of running time and solution accuracy. Wall-clock times are reported in seconds unless otherwise indicated. For the synthetic benchmarks in Section~\ref{sec:standard-synthetic}, we test one instance per $(d,m)$ and report the relative radius error $\mathrm{relerr}(\bar r):=|\bar r-r^{\mathrm{ref}}|/|r^{\mathrm{ref}}|$, where $\bar{r}$ is the returned radius and $r^{\mathrm{ref}}$ is the radius computed by CGAL.

\subsection{Synthetic benchmarks}\label{sec:standard-synthetic}

We test $d\in\{20,30,40,50\}$ and $m\in\{10^3,5\times10^3,10^4,5\times10^4,10^5\}$. Tables~\ref{tab:standard-synthetic}(a,b) report running time, accuracy, and retained fractions. Both GPU implementations solve all 20 instances in under one second, and mp-ripALM is generally faster since screening leaves a small reduced problem, as shown in Table~~\ref{tab:standard-synthetic}(b). One can also observe that both CPU implementations are faster than the conic baselines on all 20 instances. At $m=10^5$, CGAL is faster than both CPU implementations for $d=20,30$ and slower for $d=40,50$. Thus, CGAL is competitive at lower dimensions but becomes costly as $d$ increases, consistent with its dimension-dependent LP-type routine~\cite{cgal:eb-26a}. All six nonreference configurations attain relative radius errors below $10^{-8}$, while the four ripALM configurations are at least two orders of magnitude more accurate than SCS and CLARABEL for every tested $d$.

\begin{table}[!]
\caption{Synthetic benchmarks. Panel (a) reports times (seconds) and, for each $d$, the maximum relative radius error over the five values of $m$ relative to CGAL. Boldface indicates the lowest reported time for each $(d,m)$, including ties. The CPU and GPU ripALM times are three-run medians following one warm-up run; the baseline times are the recorded single-run results. Panel (b) reports the percentage retained by GPU mp-ripALM, and panel (c) reports CPU/GPU speedups at $m=10^5$. CGAL is used as the reference solver.}\label{tab:standard-synthetic}
\centering
\scriptsize
\setlength{\tabcolsep}{4pt}
\renewcommand{\arraystretch}{1.12}
\textbf{(a) Wall-clock time and accuracy}\par\smallskip
\begin{tabular*}{\textwidth}{@{\extracolsep{\fill}}clrrrrrr@{}}
\toprule
& & \multicolumn{5}{c}{Number of balls, $m$} & \shortstack{Max.\\relerr} \\
\cmidrule(lr){3-7}
$d$ & Solver & $10^3$ & $5\times10^3$ & $10^4$ & $5\times10^4$ & $10^5$ & \\
\midrule
$20$ & ripALM (GPU) & 0.048 & 0.061 & \textbf{0.053} & \textbf{0.213} & \textbf{0.292} & $4.53\times10^{-12}$ \\
& ripALM (CPU) & 0.022 & 0.066 & 0.083 & 0.657 & 1.93 & $4.53\times10^{-12}$ \\
& mp-ripALM (GPU) & 0.052 & 0.056 & 0.102 & 0.262 & 0.296 & $9.87\times10^{-13}$ \\
& mp-ripALM (CPU) & \textbf{0.021} & \textbf{0.050} & 0.054 & 0.415 & 0.986 & $9.76\times10^{-13}$ \\
& CGAL & 0.079 & \textbf{0.050} & 0.094 & 0.435 & 0.871 & Ref. \\
& SCS & 1.30 & 21.8 & 40.9 & 1054 & 2235 & $4.11\times10^{-9}$ \\
& CLARABEL & 0.809 & 14.2 & 46.2 & 603 & 2039 & $4.65\times10^{-10}$ \\
\midrule
$30$ & ripALM (GPU) & 0.157 & 0.144 & \textbf{0.153} & 0.452 & 0.508 & $1.94\times10^{-14}$ \\
& ripALM (CPU) & 0.253 & 0.198 & 0.210 & 1.55 & 3.86 & $1.94\times10^{-14}$ \\
& mp-ripALM (GPU) & 0.145 & 0.137 & 0.189 & \textbf{0.413} & \textbf{0.446} & $5.96\times10^{-13}$ \\
& mp-ripALM (CPU) & 0.302 & \textbf{0.132} & 0.187 & 0.907 & 2.10 & $6.16\times10^{-13}$ \\
& CGAL & \textbf{0.094} & 0.255 & 0.389 & 0.908 & 1.53 & Ref. \\
& SCS & 1.56 & 14.4 & 42.5 & 781 & 3750 & $4.75\times10^{-9}$ \\
& CLARABEL & 0.887 & 11.4 & 35.0 & 1036 & 3031 & $3.38\times10^{-10}$ \\
\midrule
$40$ & ripALM (GPU) & 0.192 & 0.165 & \textbf{0.179} & 0.332 & 0.543 & $5.60\times10^{-14}$ \\
& ripALM (CPU) & \textbf{0.147} & 0.199 & 0.282 & 2.10 & 3.64 & $5.60\times10^{-14}$ \\
& mp-ripALM (GPU) & 0.160 & 0.153 & 0.236 & \textbf{0.323} & \textbf{0.338} & $3.86\times10^{-13}$ \\
& mp-ripALM (CPU) & 0.152 & \textbf{0.121} & 0.185 & 1.02 & 1.89 & $3.47\times10^{-14}$ \\
& CGAL & 0.411 & 223 & 184 & 89.2 & 222 & Ref. \\
& SCS & 2.09 & 23.1 & 70.8 & 1190 & 4459 & $1.28\times10^{-9}$ \\
& CLARABEL & 0.965 & 16.0 & 59.7 & 966 & 3849 & $2.66\times10^{-10}$ \\
\midrule
$50$ & ripALM (GPU) & 0.177 & 0.205 & 0.194 & \textbf{0.351} & 0.932 & $2.56\times10^{-14}$ \\
& ripALM (CPU) & 0.187 & 0.261 & 0.373 & 2.47 & 7.26 & $2.56\times10^{-14}$ \\
& mp-ripALM (GPU) & 0.163 & 0.153 & \textbf{0.175} & 0.391 & \textbf{0.498} & $4.35\times10^{-14}$ \\
& mp-ripALM (CPU) & \textbf{0.111} & \textbf{0.143} & 0.224 & 1.59 & 3.03 & $4.35\times10^{-14}$ \\
& CGAL & 0.216 & 3599 & 3841 & 2032 & 1566 & Ref. \\
& SCS & 2.23 & 32.1 & 95.8 & 1424 & 5219 & $1.54\times10^{-9}$ \\
& CLARABEL & 1.07 & 19.2 & 72.5 & 1176 & 4610 & $3.77\times10^{-10}$ \\
\bottomrule
\end{tabular*}

\vspace{0.6em}
\begin{minipage}[t]{0.61\textwidth}
\vspace{0pt}
\centering
\textbf{(b) GPU mp-ripALM: retained constraints (\%)}\par\smallskip
\setlength{\tabcolsep}{2pt}
\begin{tabular*}{\linewidth}{@{\extracolsep{\fill}}rrrrrr@{}}
\toprule
$d$ & $10^3$ & $5\times10^3$ & $10^4$ & $5\times10^4$ & $10^5$ \\
\midrule
$20$ & 2.30 & 0.86 & 0.82 & 0.81 & 0.81 \\
$30$ & 2.80 & 0.84 & 0.88 & 0.90 & 0.90 \\
$40$ & 4.50 & 1.54 & 1.55 & 1.56 & 1.59 \\
$50$ & 4.10 & 1.92 & 1.82 & 1.86 & 1.85 \\
\bottomrule
\end{tabular*}
\end{minipage}\hfill
\begin{minipage}[t]{0.35\textwidth}
\vspace{0pt}
\centering
\textbf{(c) CPU/GPU speedup at $m=10^5$}\par\smallskip
\setlength{\tabcolsep}{3pt}
\begin{tabular*}{\linewidth}{@{\extracolsep{\fill}}rrr@{}}
\toprule
$d$ & ripALM & mp-ripALM \\
\midrule
$20$ & 6.61 & 3.33 \\
$30$ & 7.61 & 4.71 \\
$40$ & 6.70 & 5.59 \\
$50$ & 7.80 & 6.08 \\
\bottomrule
\end{tabular*}
\end{minipage}
\end{table}

\subsection{Comparison between CPU and GPU performance}\label{sec:implementation-diagnostics}

Table~\ref{tab:standard-synthetic}(c) compares CPU and GPU runs of the same PyTorch implementation and settings, where the CPU uses 16 threads. Clearly, leveraging hardware acceleration together with mixed-precision screening can substantially improve computational efficiency while maintaining high accuracy in the returned solution. Table~\ref{tab:gpu-diagnostics} reports GPU diagnostics at $m=10^5$, combining ripALM iteration counts and average wall-clock time per inner Newton iteration with a separate phase profile of mp-ripALM. We observe that the float32/full solve dominates the measured mp-ripALM runtime and screening-related overhead remains small. Because the phase timings do not isolate individual operations or synchronization, they do not identify a specific bottleneck or establish whether execution is compute-bound or memory-bound.

\begin{table}[!]
\caption{GPU diagnostics at $m=10^5$ (three-run medians). O/I gives the outer iteration count and the total inner Newton iteration count across all outer iterations; LS denotes the line-search evaluation count. For ripALM, time/inner iter. is total wall-clock time divided by the total inner Newton iteration count, expressed in milliseconds. In the FP32/full and FP64/reduced rows, each entry gives time share (\%) $|$ (O/I/LS). The unlisted remainder includes data conversion, reduced-problem setup, final KKT evaluation, and unattributed overhead.}\label{tab:gpu-diagnostics}
\centering
\footnotesize
\renewcommand{\arraystretch}{1.08}
\setlength{\tabcolsep}{2pt}
\begin{tabular*}{\textwidth}{@{\extracolsep{\fill}}lrrrr@{}}
\toprule
Metric & $d=20$ & $d=30$ & $d=40$ & $d=50$ \\
\midrule
ripALM iterations (O/I) & 5/35 & 6/50 & 6/47 & 6/61 \\
ripALM time/inner iter. (ms) & 8.34 & 10.15 & 11.55 & 15.27 \\
\addlinespace[2pt]
mp FP32/full: time (\%) $|$ (O/I/LS) & $75.4\;|\;(3/25/88)$ & $77.1\;|\;(4/40/132)$ & $73.6\;|\;(3/30/91)$ & $82.0\;|\;(4/40/130)$ \\
mp FP64/reduced: time (\%) $|$ (O/I/LS) & $22.5\;|\;(4/10/19)$ & $21.4\;|\;(5/14/40)$ & $24.3\;|\;(5/18/36)$ & $16.8\;|\;(5/16/43)$ \\
mp screening+full check: time (\%) & 0.9 & 0.6 & 0.8 & 0.5 \\
\bottomrule
\end{tabular*}
\end{table}

\subsection{Sensitivity}\label{sec:sensitivity}

We test seven screening ratios from $0.90$ to $1$ on three seeded uniform instances with $(d,m)=(50,10^5)$, HIGGS, and the boundary-dense instance. Each setting has one warm-up and three reported uninstrumented runs; separate instrumented runs measure $H$. The results are summarized in Table~\ref{tab:screening-sensitivity}. All reported runs pass the full-set check. The largest run-wise full-set violation and KKT residual are both $2.94\times10^{-9}$, and the largest relative difference between the median mp-ripALM and ripALM radii is $3.21\times10^{-12}$. With the default screening ratio $\gamma_{\rm scr}=0.98$, screening retains only a small fraction of constraints for the synthetic and HIGGS instances, and no additional refinement is needed. At $\gamma_{\rm scr}=1$, reinsertion may be required. For the boundary-dense instance, screening retains all constraints at every tested ratio, and mp-ripALM takes longer than ripALM. We additionally conduct separate one-factor tests on the deterministic $(d,m)=(50,10^5)$ instance to evaluate the impact of the low-precision KKT tolerance and the initial penalty parameter $\sigma_0$, using one warm-up and three timed runs per setting, with all other default hyperparameters in Table~\ref{tab:solver-parameters} fixed. We observe that tightening the low-precision KKT tolerance from $10^{-2}$ to $10^{-3}$ changes the median time from $0.322$ to $0.375$ s; for $\sigma_0=0.5,1,2$, the ripALM medians are $0.555$, $0.826$, and $0.712$ s.

\begin{table}[!]
\caption{Screening-ratio sensitivity (seconds). Kept is the percentage of constraints in the initial retained set, $T$ is total uninstrumented mp-ripALM time, and $H$ is high-precision phase time measured in a separate phase-instrumented run. Boundary-dense retains $100\%$ at every ratio. Speedup is ripALM/mp-ripALM; Max. ref. (maximum refinement) is ordered Synthetic/HIGGS/Boundary-dense.}\label{tab:screening-sensitivity}
\centering
\scriptsize
\setlength{\tabcolsep}{2pt}
\renewcommand{\arraystretch}{1.08}
\begin{tabular*}{\textwidth}{@{\extracolsep{\fill}}crrrrrrrrc@{}}
\toprule
& \multicolumn{3}{c}{Synthetic} & \multicolumn{3}{c}{HIGGS} & \multicolumn{2}{c}{Boundary-dense} & Max. ref. \\
\cmidrule(lr){2-4}\cmidrule(lr){5-7}\cmidrule(lr){8-9}
$\gamma_{\rm scr}$ & Kept (\%) & $T/H$ (s) & Speedup & Kept (\%) & $T/H$ (s) & Speedup & $T/H$ (s) & Speedup & S/H/B \\
\midrule
$0.900$ & $6.733$ & $0.211/0.053$ & $1.80$ & $56.919$ & $0.775/0.257$ & $1.63$ & $0.112/0.037$ & $0.73$ & $0/0/0$ \\
$0.950$ & $0.918$ & $0.204/0.060$ & $1.86$ & $1.048$ & $0.580/0.048$ & $2.18$ & $0.111/0.034$ & $0.74$ & $0/0/0$ \\
$0.980$ & $0.155$ & $0.218/0.055$ & $1.74$ & $0.0061$ & $0.594/0.048$ & $2.12$ & $0.113/0.034$ & $0.72$ & $0/0/0$ \\
$0.990$ & $0.076$ & $0.218/0.051$ & $1.74$ & $0.0015$ & $0.577/0.051$ & $2.19$ & $0.111/0.034$ & $0.73$ & $0/0/0$ \\
$0.995$ & $0.053$ & $0.212/0.057$ & $1.79$ & $0.0009$ & $0.573/0.049$ & $2.20$ & $0.109/0.034$ & $0.75$ & $0/0/0$ \\
$0.999$ & $0.038$ & $0.216/0.055$ & $1.76$ & $0.0007$ & $0.551/0.051$ & $2.29$ & $0.109/0.034$ & $0.75$ & $0/0/0$ \\
$1.000$ & $0.013$ & $0.489/0.312$ & $0.69$ & $0.0004$ & $0.653/0.135$ & $1.93$ & $0.118/0.035$ & $0.69$ & $2/1/0$ \\
\bottomrule
\end{tabular*}
\end{table}

\subsection{High-dimensional and real-data problems}\label{sec:high-dimensional}\label{sec:real-data}

Table~\ref{tab:benchmark-results} shows increasing benefits at high dimension and extends the evaluation to real data and the boundary-dense test. All nine real-data mp-ripALM runs pass the full-set check without refinement (with the max KKT residual $9\times10^{-9}$), whereas retaining every constraint makes mp-ripALM slower on the boundary-dense instance. To assess linear variance concentration, we apply approximate PCA to 10,000 sampled centers from each dataset. Explaining at least $95\%$ of the variance requires 184 of 784 components for Fashion-MNIST, 23 of 28 for HIGGS, and 69 of 784 for FEMNIST. Thus, the variance is more concentrated in a low-dimensional linear subspace for the two image datasets than for HIGGS.

\begin{table}[H]
\caption{High-dimensional and real-data GPU results (seconds). Panel (b) reports three-run medians; its HIGGS and boundary-dense values reuse the sensitivity runs. All mp-ripALM entries use $\gamma_{\rm scr}=0.98$, and speedup is ripALM/mp-ripALM. Here $r_i>0$ is the positive-radius percentage, Boundary is the percentage with relative slack at most $10^{-4}$, and Kept is the initial retained set.}\label{tab:benchmark-results}
\centering
\scriptsize
\renewcommand{\arraystretch}{1.08}
\begin{minipage}[t]{0.43\textwidth}
\vspace{0pt}
\centering
\textbf{(a) High-dimensional synthetic instances}\par\smallskip
\setlength{\tabcolsep}{2pt}
\begin{tabular*}{\linewidth}{@{\extracolsep{\fill}}rrrrrrr@{}}
\toprule
& \multicolumn{3}{c}{$m=10^3$} & \multicolumn{3}{c}{$m=5\times10^3$} \\
\cmidrule(lr){2-4}\cmidrule(lr){5-7}
$d$ & ripALM & mp-ripALM & Speedup & ripALM & mp-ripALM & Speedup \\
\midrule
$100$ & $0.198$ & $0.152$ & $1.31$ & $0.096$ & $0.094$ & $1.01$ \\
$500$ & $0.180$ & $0.173$ & $1.04$ & $0.680$ & $0.363$ & $1.87$ \\
$1000$ & $0.487$ & $0.403$ & $1.21$ & $1.17$ & $0.469$ & $2.50$ \\
$5000$ & $11.48$ & $5.54$ & $2.07$ & $15.53$ & $7.57$ & $2.05$ \\
$10000$ & $79.60$ & $31.00$ & $2.57$ & $281.34$ & $96.66$ & $2.91$ \\
\bottomrule
\end{tabular*}
\end{minipage}\hfill
\begin{minipage}[t]{0.53\textwidth}
\vspace{0pt}
\centering
\textbf{(b) Real-data and boundary-dense instances}\par\smallskip
\setlength{\tabcolsep}{2pt}
\begin{tabular*}{\linewidth}{@{\extracolsep{\fill}}lrrrrrr@{}}
\toprule
Instance & $r_i>0$ & Boundary & Kept & ripALM & mp-ripALM & Speedup \\
\midrule
Fashion-MNIST & 0\% & 0.045\% & 171 (0.285\%) & 1.982 & 0.666 & 2.98 \\
HIGGS & 0\% & 0.0005\% & 61 (0.0061\%) & 1.261 & 0.594 & 2.12 \\
FEMNIST & 99.997\% & 0.0052\% & 4 (0.0052\%) & 1.914 & 0.731 & 2.62 \\
\addlinespace[2pt]
Boundary-dense & 100\% & 100\% & 100,000 (100\%) & 0.0817 & 0.1130 & 0.72 \\
\bottomrule
\end{tabular*}
\end{minipage}
\end{table}

\section{Conclusion}\label{sec:conclusion}

This paper developed a GPU-oriented ripALM framework for MEB based on an SOCP reformulation. Exploiting constraint-wise separability turns objective, derivative, and multiplier evaluations into independent per-ball operations followed by parallel reductions, while the relative-type stopping rule avoids oversolving and retains the convergence guarantees of ripALM. We also introduced mp-ripALM, which uses low precision for screening and warm starting, solves the reduced problem in high precision, and verifies feasibility over all original constraints with reinsertion when needed. Experiments show high accuracy and strong performance relative to CGAL and general conic solvers on the tested instances, with mp-ripALM accelerating ripALM when screening is effective. The current implementation uses dense Newton solves on a single GPU, and mixed-precision reduction can add overhead when screening retains all constraints. Future work includes sharper screening guarantees, adaptive first-order inner solves with explicit complexity guarantees, matrix-free linear algebra, and multi-GPU or distributed implementations for problems whose data exceed the memory capacity of a single GPU.

\section*{Acknowledgments}

\noindent The authors thank editor and referees for their insightful comments and constructive suggestions, which have helped improve the paper. Ling Liang acknowledges support from the Department of Mathematics at the University of Tennessee, Knoxville, through a summer research fellowship. The research of Lei Yang is supported in part by the National Natural Science Foundation of China under grant 12301411, and the Natural Science Foundation of Guangdong under grant 2026A1515010350.

\bibliographystyle{ACM-Reference-Format}
\bibliography{Refs_MEB}


\end{document}